\documentclass[12pt]{article}
\usepackage[english]{babel}

\usepackage[letterpaper,top=2cm,bottom=2cm,left=3cm,right=3cm,marginparwidth=1.75cm]{geometry}
\usepackage{amsmath,amssymb,amsthm,amscd,esint}
\usepackage{tikz-cd}

\usepackage{graphicx}
\usepackage[colorlinks=true, allcolors=blue]{hyperref}
\newtheorem{theorem}{Theorem}[section]
\newtheorem{lemma}{Lemma}[section]
\newtheorem{corollary}[theorem]{Corollary}
\newtheorem{proposition}[theorem]{Proposition}
\newtheorem{definition}[theorem]{Definition}

\newtheorem{example}{Example}
\newtheorem{remark}[theorem]{Remark}

\newcommand{\R}{\mathbb{R}}
\newcommand{\C}{\mathbb{C}}
\newcommand{\ii}{\sqrt{-1}}
\newcommand{\reg}{\mathrm{reg}}

\newcommand{\bbm}{\begin{bmatrix}}
\newcommand{\ebm}{\end{bmatrix}}
\newcommand{\lam}{\lambda}

\DeclareMathOperator{\codim}{codim}

\title{Generalized Gauduchon metrics on compact complex spaces
}
\author{
Yong Chen\thanks{Email:  cybwv988@163.com }\\
College of Field Engineering, Army Engineering University of PLA,\\ NanJing 210007,
People’s Republic of China\\
Xi Zhang\thanks{ Email:  mathzx@njust.edu.cn}\\
School of Mathematics and Statistics, \\Nanjing University of Science and Technology, Nanjing \\210094, P. R. China}
\begin{document}
\maketitle

\begin{abstract}
Let $X$ be an irreducible reduced compact complex space of dimension $n$, and let $\omega$ be a Hermitian metric on $X$. For each $1\le k\le n-1$, we prove that there exist a  constant $c_k$ and a smooth positive function $\rho$ on $X_{\reg}$ such that $\ii\partial\bar\partial(\rho\omega^k)\wedge\omega^{n-k-1}=c_k\rho\omega^n$. The function $\rho$ is bounded above and bounded away from zero. The constant $c_k$ is unique, and $\rho$ is unique up to a positive multiple among bounded smooth positive solutions. We also prove that the full sign set $\{-1,0,1\}$ of the generalized Gauduchon constants $c_k$  on a compact complex manifold is preserved by finite sequences of blow-ups with smooth centers.
\end{abstract}

\section{Introduction}

Let $X$ be a $n$-dimensional compact complex manifold and $g$ be a Riemmann metric which satisfies $g(JX,JY)=g(X,Y)$ (compatible with the complex structure $J$). This metric $g$ is called a Hermitian meric. We can  identify $g$ with a real positive definite $(1,1)$-form defined by
 \[ \omega(X,Y):=g(JX,Y)  \]
 for any $X,Y\in \Gamma(TX)$.  Gauduchon \cite{Gau} proved an interesting result that, for any Hermitian metric $\omega$ on a compact complex $n$-dimensional manifold $X$, there exists a Hermitian metric $\omega_{G}$ in the conformal class of $\omega$, i.e., $\omega_{G}=\rho^{\frac{1}{n-1}}\omega$ where $\rho$ is a positive function on $X$, such that
\begin{equation}\label{GauE}
    \sqrt{-1}\partial\bar{\partial}\omega^{n-1}_{G}=\sqrt{-1}\partial\bar{\partial}(\rho\omega^{n-1})=0,
\end{equation}
 $\omega_{G}$ is unique up to a positive multiple. Thus, a compact complex manifold always admits a Gauduchon metric. This fact is important in complex geometry since one can use such a metric to define the degree of holomorphic vector bundles on complex manifolds, and then make sense of the stability of holomorphic vector bundles on a non-K\"ahler  manifold, as a consequence, Li and Yau generalized the Donaldson-Uhlenbeck–Yau theorem to non-K\"ahler setting (see \cite{LY}). 
 
 For $1\le k\le n-2$, one may ask whether a complex manifold admits a Hermitian metric $\omega$ satisfying
\begin{equation}  \label{GG0} \sqrt{-1}\partial\bar{\partial}\omega^{k}=0.
\end{equation}
When $k = n - 2$, a Hemitian metric $\omega$ satisfying \eqref{GG0}
 is called an Astheno-K\"ahler metric
 which was introduced by 
 Jost and Yau \cite{YY}  to study Hermitian harmonic maps.  When $k = 1$, the metric $\omega$ in \eqref{GG0} is called pluriclosed, or SKT (see \cite{FPS, ST}). However, there exist complex manifolds which don't admit an Astheno-K\"ahler metric or a pluriclosed metric. Astheno-K\"ahler metrics and pluriclosed metrics impose some constraints on the underlying manifolds. 
  Since equation \eqref{GG0} does not necessarily have a solution in complex manifolds, Fu, Wang and Wang \cite{FWW} considered the
weaker equation :
  \begin{equation}\label{k-Gau}
\sqrt{-1}\partial\bar{\partial}\omega^{k}\wedge\omega^{n-k-1}=0
  \end{equation}
  for $1\le k \le n-1$,
such a Hermitian metric is called  a $k$-Gauduchon metric. Note that a $(n-1)$-Gauduchon metric is the classical Gauduchon metric. To study the existence of $k$-Gauduchon metrics, Fu, Wang and Wang \cite[Theorem 2, Corollary 4]{FWW} proved the following theorem

\begin{theorem}\label{FWWMT}
Let $(M, \omega)$ be a compact Hermitian manifold. For $1\le k \le n-1$. The following equation
\begin{equation}\label{FWWG}    
\sqrt{-1}\partial\bar{\partial}(e^{v}\omega^{k})\wedge\omega^{n-k-1}=c_{k}e^{v}\omega^{n}.
\end{equation}
admits a  constant $c_k$ and a smooth function satisfying \eqref{FWWG}. Futhermore, the constant $c_k$ is unique and $v$  is unique up to a constant. In particular, when $k=n-1$, we have $c_{n-1}=0$. If $\omega$ is a K\"ahler metric, then $c_{k}=0$ and $v$ is a constant for $1\le k \le n-1$.
\end{theorem}


 Our convention uses $\sqrt{-1}\partial\bar\partial$, rather than $(\sqrt{-1}/2)\partial\bar\partial$ as in \cite{FWW}. This change does not affect the sign of the constant. Equation \eqref{FWWG} is equivalent to 
\begin{equation}\label{quasi}    \begin{split}        
\Delta v+|\partial v|^2+\langle B_{k,\omega},\mathrm{d} v\rangle+\varphi_{k,\omega}=nc_k,
\end{split}
\end{equation}
where $B_{k,\omega}=\sqrt{-1}\frac{k}{(n-1)(n-1)!}*(\partial\omega^{n-1}-\bar{\partial}\omega^{n-1})$ and $\varphi_{k,\omega}=\frac{n\sqrt{-1}\partial\bar{\partial}\omega^{k}\wedge\omega^{n-k-1}}{\omega^{n}}$. If we set $\rho=e^v$, equation \eqref{quasi} are turned  into a linear elliptic equation for $\rho$.

Fu, Wang, and Wang employed the classical continuity method to solve equation \eqref{quasi}, a second-order quasi-linear PDE.
The key to this approach is establishing a priori estimates for the solution $v$. They first derived a global gradient estimate that depends on the Hermitian metric $\omega$ by applying the maximum principle. Subsequently, a $C^{0}$-estimate follows easily from this gradient bound (see \cite[page 665]{FWW} for details). Finally, applying the interior $W^{2,p}$-estimates ($p>1$) and Schauder estimates for second-order linear elliptic equations yield the desired global $C^{k,\alpha}$-estimates for any $k\ge1$. 

Singular spaces, especially projective varieties are main research subject in algebraic geometry and analytic geometry. Singularities are unavoidable as they occur in various contexts, especially in the minimal model program and moduli theory. There has been increasing attention in recent years on the study of canonical metrics on singular varieties. Eyssidieux, Guedj and Zeriahi \cite{EGZ} constructed singular K\"ahler–Einstein metrics with non-positive curvature on mildly projective varieties by solving a degenerate complex Monge-Amp\`ere equation. The Donaldson–Uhlenbeck–Yau theorem was generalized to normal K\"ahler spaces (see \cite{CGP, CW, CZ, Pan}).  Pan \cite{Pan1} first constructed bounded Gauduchon metrics on
smoothable Hermitian varieties (complex spaces with Hermitian metrics). He later removed the smoothability assumption \cite{Pan}.  An important part of the latter work is a uniform Sobolev inequality for perturbations of a pullback metric on a resolution. It uses the recent Sobolev estimates of Guo, Phong, Song and Sturm \cite{GPSS}.

We consider corresponding problem for generalized Gauduchon metrics on compact complex spaces. The meaning of a Hermitian metric on a complex space is specified in Definition~\ref{Herm}. 

\begin{definition}
Let $X$ be an irreducible reduced compact complex space of  dimension $n$, and let $\omega$ be a Hermitian metric on $X$. Fix $1\le k\le n-1$. If there exists a  positive bounded smooth function defined on $X_{\mathrm{reg}}$ such that
\begin{equation}
\sqrt{-1}\partial\bar{\partial}(\rho\omega^{k})\wedge\omega^{n-k-1}=0
\end{equation}
holds on $X_{\mathrm{reg}}$. Then $\widetilde{\omega}=\rho^{\frac{1}{k}}\omega$ is a Hermitian metric on $X_{\mathrm{reg}}$ satisfying
\begin{equation}
    \sqrt{-1}\partial\bar{\partial}(\widetilde{\omega}^{k})\wedge\widetilde{\omega}^{n-k-1}=0
\end{equation}
Then we call $\widetilde{\omega}$ a bounded $k$-Gauduchon metric in the conformal class of $\omega$.
\end{definition}

The conformal factors constructed below are also bounded away from zero. The corresponding metrics are therefore uniformly equivalent to the background metric on the regular locus. Our first result extends the existence and uniqueness statements in Theorem~\ref{FWWMT} to singular spaces.

\begin{theorem}\label{MT}
Let $X$ be an irreducible reduced compact complex space of dimension $n$, and let $\omega$ be a Hermitian metric on $X$. For each $1\le k\le n-1$, there exist a  constant $c_k$ and a smooth positive function  $\rho\in C^{\infty}(X_{\mathrm{reg}})\cap L^{\infty}(X_{\mathrm{reg}})$ satisfying 
\begin{equation}\label{SGG}
\ii\partial\bar\partial(\rho\omega^k)\wedge\omega^{n-k-1}=c_k\rho\omega^n
\end{equation}
and
\[
0<\inf_{X_{\reg}}\rho\le\sup_{X_{\reg}}\rho<\infty.
\]
Among bounded smooth positive solutions of \eqref{SGG}, the constant $c_k$ is unique and $\rho$ is unique up to multiplication by a positive constant. In particular, the normalization $\inf_{X_{\mathrm{reg}}}\rho=1$ determines $\rho$ uniquely. For $k=n-1$, the constant is $c_{n-1}=0$. If $\omega$ is K\"ahler, then for $1\le k\le n-1$,  $c_k=0$ and $\rho$ is constant.
\end{theorem}

\begin{remark}
    By the uniqueness of $c_k$ for equation \eqref{SGG}, we can associate to each hermitian metric $\omega$ on an irreducible reduced compact complex space a unique constant $c_k(\omega)$
\end{remark}

Now, we  outline the approach to proving Theorem \ref{MT}. Consider a resolution  $\pi: \widetilde{X} \to X$ of the singularities of $X$. Denote by $D$ the exceptional divisor of $\pi$.    The map $\pi$ is an isomorphism away from the singular locus of $X$. We can  construct a family of Hermitian metrics  $\omega_{\epsilon}$ ($\epsilon\in (0,1]$) on $\widetilde{X}$ which converge to $\pi^{*}\omega$ in $C_{\mathrm{loc}}^{\infty}(\widetilde{X}\backslash D)$ as $\epsilon \to 0$. Thanks to Theorem \ref{FWWMT}, for each $\epsilon$, there exists a unique smooth positive function $\rho_\epsilon$ and $c_{k,\epsilon}$ satisfying the equation
\begin{equation}\label{f_k-Gau}
    \sqrt{-1}\partial\bar{\partial}(\rho_{\epsilon} \omega^{k}_{\epsilon})\wedge\omega_{\epsilon}^{n-k-1}=c_{k,\epsilon}\rho_{\epsilon}\omega_{\epsilon}^{n},\;\;\inf\rho_{\epsilon}=1.
\end{equation}
To establish the convergence of $\rho_\epsilon$ as $\epsilon \to 0$, we require a uniform estimate for $\rho_\epsilon$.  

We consider the following equation 
\begin{equation}
\sqrt{-1}\partial\bar{\partial}(\rho \omega^{k})\wedge\omega^{n-k-1}=c_{k}\rho\omega^{n},\;\;\inf\rho=1
\end{equation}
on  Hermitian manifold $(M, \omega)$.
Let $A$ be a positive constant which depends on  Hermitian metric $\omega$ such that the metric $\omega$ satisfies the following inequalities:
\begin{equation}\label{1.5}
    -A\omega^{n}\le \sqrt{-1}\partial\bar{\partial}\omega^{k}\wedge\omega^{n-k-1}\le A\omega^{n}
\end{equation}
\begin{equation}\label{1.7}
    -A\omega^{n}\le \sqrt{-1}\partial\bar{\partial}\omega^{n-k-1}\wedge\omega^{k}\le A\omega^{n}
\end{equation}

In Section \ref{C00}, we establish a  global $L^{\infty}$-estimate for $\rho$. This estimate depends on the parameters $n$, $A$, $C_S$, $C_P$, $V_\omega$, $|c_k|$, and $\beta$. Here, $C_S$ and $C_P$ are the Sobolev and Poincar\'e constants of the Hermitian metric $\omega$, respectively, $V_\omega$ is the volume of $\omega$, and $\beta \in (1, \frac{n}{n-1})$ is a constant.  The proof uses positive and negative powers of $\rho$. A small positive power gives an energy coefficient that tends to zero. This allows us to absorb the Poincar\'e term and complete the estimate. The local $C^{k,\alpha}$-estimates for any $k\ge 1$ come from the interior $W^{2,p}$-estimates ($p>1$) and Schauder estimates for elliptic equations.

 In Section \ref{ext}, we prove that the constants $A_\epsilon$ appearing in inequalities \eqref{1.5} and \eqref{1.7} for the metrics $\omega_\epsilon$ are independent of $\epsilon$. Lemma \ref{uniform} establishes that $|c_{k,\epsilon}|$ is uniformly bounded for all $\epsilon \in (0,1]$. Furthermore, uniform Sobolev inequalities and Poincar\'e inequalities for $\omega_\epsilon$ have recently been established by Pan  \cite[Theorem 3.5]{Pan}.
 Combining these results, we obtain a  global uniform   $L^{\infty}$-estimate for $\rho_\epsilon$. Applying the interior $W^{2,p}$-estimates ($p>1$) and Schauder estimates for elliptic equations, we obtain a local uniform $C^{k,\alpha}$-estimate for any $k\ge 1$. Consequently, we can select a subsequence $\rho_{\epsilon_{i}}$ such that $\rho_{\epsilon_{i}}$ converges smoothly (away from the singular locus of $X$) to a limit function $\rho$ and $c_{k,\epsilon_{i}}$ converges to a  constant $c_k$ as $i \to \infty$. Obviously, $\rho$ and $c_k$ satisfies equation \eqref{SGG} on $X_{\mathrm{reg}}$. At the end of Section \ref{ext}, We give  examples of  three-dimensional non-K\"ahler varieties  that admits a 1-Gauduchon metric.

In Section~\ref{uniq}, we define the linear operator $\mathcal L_{k,\omega}$ motivated by equation \eqref{SGG}. Uniqueness follows from the formal adjoint relation ($\mathcal L_{k,\omega}^*=\mathcal L_{n-k-1,\omega}$) and cutoff estimates near the singular locus. We first prove that a bounded solution of the adjoint equation has finite energy. Cutoff estimates then give uniqueness of the constant and uniqueness of the positive function up to a positive multiple. The same argument gives $c_k(\omega)=c_{n-k-1}(\omega)$ for $1\le k\le n-2$ (see Corollary~\ref{cor:adjoint-symmetry}). 

For a compact  complex manifold $M$, let $c_k(\omega)$ be the unique constant in Theorem~\ref{FWWMT}, and set
\[
C_k(M)=\{\mathrm{sgn}\; c_k(\omega):\omega\text{ is a Hermitian metric on }M\}\subset\{-1,0,1\}.
\]
The constant depends smoothly on $\omega$ by \cite[Proposition 9]{FWW}. 
Hence, if $-1, 1\in C_k(M) $, then $C_k(M)=\{-1,0,1\}$.
Also, $c_k(\omega)=0$ precisely when the conformal class of $\omega$ contains a $k$-Gauduchon metric. Fu, Wang and Wu asked (\cite[page 662]{FWW})  whether $C_k(M)$ is invariant under modifications. Our second result gives one direction when the sign set is full.

\begin{theorem}\label{thm2}
Let $X$ and $Y$ be compact connected complex manifolds of dimension $n$. Suppose that there is a finite sequence of blow-ups with smooth centers
\[
X=X_N\xrightarrow{f_{N-1}}X_{N-1}\longrightarrow\cdots\longrightarrow X_2\xrightarrow{f_1}X_1=Y.
\]
For $1\le k\le n-2$, 
\begin{itemize}
    \item if $-1\in C_k(Y)$, then  $-1\in C_k(X)$.
    \item if $1\in C_k(Y)$, then  $1\in C_k(X)$.
    \item if $C_k(Y)=\{-1,0,1\}$, then $C_k(X)=\{-1,0,1\}$.
\end{itemize}
\end{theorem}

For a single blow-up $f:X\to Y$, we prove more precisely that $c_k(f^*\omega+\epsilon\theta)\to c_k(\omega)$ for a suitable Hermitian metric $\theta$ on $X$. The limit equation is transferred to the complement of the center in $Y$. 
A extension lemma for solution of  the limit equation is  then applied on $Y$. The proof of Theorem \ref{thm2} is given in Section~\ref{blowups}.

\section{Notation and preliminaries}

Let $X$ be a $n$-dimensional complex manifold and $\omega=\sqrt{-1}g_{i\bar{j}}\mathrm{d}z^{i}\wedge\mathrm{d}\bar{z}^{j}$ be a Hermitian metric on $X$. For any real $(1,1)$-form $\alpha=\sum_{i=1}^{n}(\alpha_{i}\mathrm{d}z^{i}+\alpha_{\bar{i}}\mathrm{d}\bar{z}^{i})$ and $(1,1)$-form $\beta =\sum_{i=1}^{n}(\beta_{i}\mathrm{d}z^{i}+\beta_{\bar{i}}\mathrm{d}\bar{z}^{i})$. The inner product induced by $g$ on $\Lambda T^{*}M$ is
\begin{equation}
    \begin{split}
\langle\alpha,\beta\rangle=\sum_{i,j=1}^{n}g^{i\bar{j}}(\alpha_{i}\beta_{\bar{j}}+\alpha_{\bar{i}}\beta_{j}).
    \end{split}
\end{equation}
For a smooth function $u$ on $X$,  this gives
\begin{equation}\label{eq:norm-convention}
|\partial u|^2=\sum_{i,j}g^{i\bar j}u_i u_{\bar j},\qquad |\mathrm{d}u|^2=2|\partial u|^2,\qquad \ii\partial u\wedge\bar\partial u\wedge\omega^{n-1}=\frac1n|\partial u|^2\omega^n.
\end{equation}
The complex Laplacian associated with the metric $g$ is given by
\begin{equation}
    \Delta u= \frac{n\sqrt{-1}\partial\bar{\partial}u\wedge\omega^{n-1}}{\omega^{n}}=\sum_{i,j}g^{i\bar{j}}\nabla_{i}\nabla_{\bar{j}}u=\sum_{i,j}g^{i\bar{j}}\partial_{i}\bar{\partial}_{\bar{j}}u
\end{equation}
In particular, $\Delta e^u=e^u(\Delta u+|\partial u|^2)$.

Let $X$ be an irreducible reduced complex  space of dimension $n \ge 1$. A chart $(U,\tau,V)$ on $X$ is a triple consisting of an open set $U\subset X$, a closed complex space $V\subset G\subset \mathbb{C}^{N}$ in an open set $G$ of $\mathbb{C}^{N}$ and a biholomorphic map  $\tau:U \to V$ (in the category of complex spaces). The map $\tau:U \to G\subset \mathbb{C}^{N}$ is called a local embedding of the complex space $X$. A complex space can be partitioned into two sets of points: the smooth points denoted by $X_{\mathrm{reg}}$ and the singular points denoted by $X_{\mathrm{sing}}$.

\begin{definition}\label{Herm}
A Hermitian metric on $X$ is a positive real $(1,1)$-form $\omega$ on $X_{\mathrm{reg}}$ with the following property. For every $x\in X$, there are a local embedding $j:U\hookrightarrow V\subset\mathbb{C}^N$ and a smooth Hermitian form $\Omega$ on $V$ such that $\omega|_{U\cap X_{\mathrm{reg}}}=j^*\Omega$. A K\"ahler metric on $X$ is one for which these local extensions $\Omega$ can be chosen K\"ahler.   
\end{definition}
A Hermitian metric on a paracompact complex space  can be constructed  by a partition of unity (see \cite[Section 2]{CMM}). All integrals on a singular space are taken over its regular locus. 

Pan has recently established uniform Sobolev inequalities for the perturbation path of the pullback metric on the desingularization of Hermitian varieties. We use the following uniform inequalities.

\begin{theorem}[Pan {\cite[Theorem 3.5]{Pan}}]\label{UniSo}
Let $X$ be a complex space endowed with a Hermitian metric $\omega$, and let $\pi: \widetilde{X} \to X$ be a resolution of singularities. Fix a Hermitian metric $\widetilde{\omega}$ on $\widetilde{X}$ and consider the family of Hermitian metrics defined by $\omega_{\epsilon} = \pi^{*}\omega + \epsilon \widetilde{\omega}$ for $\epsilon \in (0,1]$. Then the following properties hold:
\begin{enumerate}
\item For any $q \in \big(1, \frac{n}{n-1}\big)$ and every smooth function $f$, there exists a positive constant $C_S$, independent of $\epsilon \in (0,1]$, such that
\begin{equation}
\left( \int_{\widetilde{X}} |f|^{2q} \omega_{\epsilon}^{n} \right)^{\frac{1}{q}} \le C_S \left( \int_{\widetilde{X}} |f|^2 \omega_{\epsilon}^{n} + \int_{\widetilde{X}} |\partial f|^2 \omega_{\epsilon}^{n} \right).
\end{equation}
\item There exists a constant $C_P > 0$, independent of $\epsilon \in (0,1]$, such that for  every smooth function f
\begin{equation}\label{Poin}
\int_{\widetilde{X}} |f-\bar{f}_\epsilon|^2 \omega_{\epsilon}^{n} \le C_P \int_{\widetilde{X}} |\partial f|^2 \omega_{\epsilon}^{n}.
\end{equation}
where $\bar{f}_\epsilon=\frac{\int_{\widetilde{X}} f  \omega_{\epsilon}^{n} }{\int_{\widetilde{X}}   \omega_{\epsilon}^{n} }$
\end{enumerate}
\end{theorem}

\section{Uniform a priori estimates}\label{C00}
Let $(X,\omega)$ be a $n$-dimensional Hermitian manifold. Due to the result in \cite{FWW}, the following equation 
\begin{equation}\label{k-Gau}
    \sqrt{-1}\partial\bar{\partial}(\rho \omega^{k})\wedge\omega^{n-k-1}=c_{k}\rho\omega^{n},\;\;\inf\rho=1
\end{equation}
admits a unique smooth function $\rho$ and a constant $c_k$.

\subsection{A priori $L^{\infty}$ estimates}

Using Moser iteration, Pan \cite{Pan} established a uniform upper bound for the solution $\rho$ of equation \eqref{GauE}. This bound depends on the Sobolev and Poincar\'e constants.  We can modify Pan's method to obtain a priori $L^{\infty}$ estimates for equation \eqref{k-Gau}.

\begin{lemma}\label{FL}
    Let $f(x), s(x), F(x)$ be smooth functions defined on $\mathbb{R}^{+}$ which satisfy $F^{'}(x)=f(x), s(x)=xf(x)-F(x)$.
 Then we have
    \begin{equation}\label{GG}
    \begin{split}
        \int_{X}f^{'}(\rho)\sqrt{-1}\partial \rho\wedge \bar{\partial}\rho\wedge\omega^{n-1} &=  \int_{X}  s(\rho)\sqrt{-1}\partial\bar{\partial}\omega^{k}\wedge\omega^{n-k-1}+ \int_{X} F(\rho) \omega^{k}\wedge\\&\sqrt{-1}\partial\bar{\partial}\omega^{n-k-1} -c_{k}\int_{X}\rho f(\rho)\omega^{n} 
    \end{split}
\end{equation}
\end{lemma}
\begin{proof}
By \eqref{k-Gau}, we have
        \begin{equation}\label{1.1}
        c_{k}\int_{X}\rho f(\rho)\omega^{n}=\int_{X}f(\rho)\sqrt{-1}\partial\bar{\partial}(\rho\omega^{k})\wedge\omega^{n-k-1}.
    \end{equation}
By Stokes’ theorem, it follows that
    \begin{equation}\label{1.2}
        \begin{split}
            \int_{X}f(\rho)&\sqrt{-1}\partial\bar{\partial}(\rho\omega^{k})\wedge\omega^{n-k-1}= -\int_{X} f^{'}(\rho) \sqrt{-1}\partial \rho\wedge \bar{\partial}(\rho\omega^{k})\wedge\omega^{n-k-1}
            + \\& \int_{X}f(\rho)\sqrt{-1}\bar{\partial}(\rho\omega^{k})\wedge\partial\omega^{n-k-1} \\&
            =-\int_{X}f^{'}(\rho)\sqrt{-1}\partial \rho\wedge \bar{\partial}\rho\wedge\omega^{n-1} - \int_{X}\rho f^{'}(\rho)\sqrt{-1}\partial \rho\wedge \bar{\partial}\omega^{k}\wedge\omega^{n-k-1}\\&+\int_{X}f(\rho)\sqrt{-1}\bar{\partial} \rho\wedge \omega^{k}\wedge\partial\omega^{n-k-1}+\int_{X}\rho f(\rho)\sqrt{-1}\bar{\partial}\omega^{k}\wedge\partial\omega^{n-k-1}.
        \end{split}
    \end{equation}
    Since $s(x)=xf(x)-F(x)$ and $F^{'}(x)=f(x)$, we have $s^{'}(x)=xf^{'}(x)$,  apply integration by parts 
    \begin{equation}\label{1.3}
        \begin{split}
            \int_{X}\rho f^{'}(\rho)&\sqrt{-1}\partial \rho\wedge \bar{\partial}\omega^{k}\wedge\omega^{n-k-1}= \int_{X} \sqrt{-1}\partial s(\rho)\wedge\bar{\partial}\omega^{k}\wedge\omega^{n-k-1} \\&= - \int_{X}  s(\rho)\sqrt{-1}\partial\bar{\partial}\omega^{k}\wedge\omega^{n-k-1}+ \int_{X} s(\rho)\sqrt{-1}\bar{\partial}\omega^{k}\wedge\partial\omega^{n-k-1}.
        \end{split}
    \end{equation}
If $F^{'}(x)=f(x)$, it follows that
    \begin{equation}\label{1.4}
        \begin{split}
\int_{X}f(\rho)&\sqrt{-1}\bar{\partial} \rho\wedge \omega^{k}\wedge\partial\omega^{n-k-1}=\int_{X}\sqrt{-1}\bar{\partial} F(\rho)\wedge \omega^{k}\wedge\partial\omega^{n-k-1} \\&= -\int_{X} F(\rho)\sqrt{-1}\bar{\partial} \omega^{k}\wedge\partial\omega^{n-k-1} + \int_{X} F(\rho) \omega^{k}\wedge\sqrt{-1}\partial\bar{\partial}\omega^{n-k-1}.
        \end{split}
    \end{equation}
Combining  this with \eqref{1.1}, \eqref{1.2} and \eqref{1.3}, we obtain \eqref{GG}.

\end{proof}

\begin{lemma}\label{iden}
For every  $p\ne0,1$, we have
\begin{equation}\label{1.611}
    \begin{split}
        \int_{X}|\partial \rho^{\frac{p}{2}}|^{2}\omega^{n}&= 
        \frac{np}{4}\int_{X}\rho^{p}\sqrt{-1}\partial\bar{\partial}\omega^{k}\wedge\omega^{n-k-1}+\frac{pn}{4(p-1)}\int_{X}\rho^{p}\sqrt{-1}\partial\bar{\partial}\omega^{n-k-1}\wedge\omega^{k}\\&-\frac{nc_{k}p^{2}}{4(p-1)}\int_{X}\rho^{p}\omega^{n}
    \end{split}
\end{equation}
\end{lemma}
\begin{proof}
    We take
\[ f(x)=\frac{p^{2}}{4(p-1)}x^{p-1},\qquad F(x)=\frac{p}{4(p-1)}x^{p},\qquad s(x)=\frac{p}{4}x^{p} \]
 For $p\ne 0,1$,  applying Lemma \ref{FL}  yields \eqref{1.611}
\end{proof}

\begin{theorem}\label{C0}
Let $(X, \omega)$ be a  $n$-dimensional compact Hermitian manifold. Suppose \eqref{1.5} and \eqref{1.7} holds for Hermitian metric $\omega$. If $(\rho,c_k)$ is the solution of \eqref{k-Gau}. 
Then we have the following uniformly $L^{\infty}$-estimates
     \begin{equation}
         \sup_{X} \rho \le C,
     \end{equation}
     where C depends on
     $n,A,C_S, C_P,V_\omega,|c_k|$ and $\beta$.
\end{theorem}

\begin{proof}
 For $p\ge 2$,  applying Lemma \ref{iden} with \eqref{1.5}  and \eqref{1.7} yields
\begin{equation}\label{1.61}
    \begin{split}
        \int_{X}|\partial \rho^{\frac{p}{2}}|^{2}\omega^{n}&= \frac{np^{2}}{4}\int_{X}\rho^{p-2}\sqrt{-1}\partial\rho\wedge\bar{\partial}\rho\wedge\omega^{n-1}\\&
        =\frac{np}{4}\int_{X}\rho^{p}\sqrt{-1}\partial\bar{\partial}\omega^{k}\wedge\omega^{n-k-1}+\frac{pn}{4(p-1)}\int_{X}\rho^{p}\sqrt{-1}\partial\bar{\partial}\omega^{n-k-1}\wedge\omega^{k}\\&-\frac{nc_{k}p^{2}}{4(p-1)}\int_{X}\rho\omega^{n}
        \\& \le \int_{X}\left(\frac{Anp}{4}\rho^{p}+\frac{np}{4(p-1)}A\rho^{p}+\frac{n|c_{k}|p^{2}}{4(p-1)}\rho^{p}\right)\omega^{n}\\&
        \le Dp\int_M\rho^p\omega^{n},
    \end{split}
\end{equation}
where $D=1+n(A+|c_k|)$ is sufficient.
Combining Sobolev inequality 
with \eqref{1.61}, we have
\begin{equation}
    \begin{split}
        \left(\int_{X}(\rho^{\frac{p}{2}})^{2\beta}\omega^{n}\right)^{\frac{1}{\beta}}&\le C_S\left(\int_{X}\rho^{p}\omega^{n}+ \int_{X}|\partial \rho^{\frac{p}{2}}|^{2}\omega^{n} \right)\\& \le C_S(1+Dp)\int_{X}\rho^{p}\omega^{n}.
    \end{split}
\end{equation}
where $\beta\in(1,\frac{n}{n-1})$.
It follows that 
\begin{equation}
    \begin{split}
        ||\rho||_{L^{\beta p}}\le \bigl(C_S(1+Dp)\bigr)^{1/p}||\rho||_{L^{p}},\;p\ge 2.
    \end{split}
\end{equation}
Taking $p_i=2\beta^i$ and iterating gives
\begin{equation}\label{1.73}
\|\rho\|_{L^\infty}\le K_1\|\rho\|_{L^2},\qquad K_1=\prod_{i=0}^\infty\bigl(C_S(1+2D\beta^i)\bigr)^{\frac{1}{2\beta^i}}.
\end{equation}
The product is finite because $\sum_i\beta^{-i}(1+i)<\infty$.

We choose $f(x)=-\frac{p^{2}}{4(p+1)}x^{-p-1}$, so that 
$f^{'}(x)=\frac{p^2}{4}x^{-p-2}$, $s(x)=-\frac{p}{4}x^{-p}$ and $F(x)=\frac{p}{4(p+1)}x^{-p}$. For $p\ge 1$, applying \eqref{GG}, this gives
\begin{equation}\label{1.63}
    \begin{split}
        \int_{X}|\partial\rho^{-\frac{p}{2}}|^{2}\omega^{n}& 
        =\frac{np^{2}}{4}\int_{X}\rho^{-p-2}\sqrt{-1}\partial\rho\wedge\bar{\partial}\rho\wedge\omega^{n-1}
        \\&=
          -\frac{np}{4}\int_{X}  \rho^{-p}\sqrt{-1}\partial\bar{\partial}\omega^{k}\wedge\omega^{n-k-1}+
           \frac{np}{4(p+1)}\int_{X} \rho^{-p} \sqrt{-1}\partial\bar{\partial}\omega^{n-k-1}\wedge\omega^{k}\\&
           +\frac{nc_{k}p^{2}}{4(p+1)}\int_{X}\rho^{-p}\omega^{n}\\& \le 
           \frac{npA}{4}\int_{X}\rho^{-p}\omega^{n}+\frac{nA}{4}\int_{X}\rho^{-p}\omega^{n}+\frac{n|c_{k}|p}{4}\int_{X}\rho^{-p}\omega^{n}
           \\& \le Dp\int_{X}\rho^{-p}\omega^{n}.
    \end{split}
\end{equation}
Combining Sobolev inequality with \eqref{1.63}, apply Moser's iteration argument with $p_i=\beta^i$, we have
\begin{equation}\label{eq:negative-Moser}
 1=\sup_{X} \rho^{-1}\le K_{2}\int_{X}\rho^{-1}\omega^{n},\qquad K_2=\prod_{j=0}^\infty\bigl(C_S(1+D\beta^j)\bigr)^{1/\beta^j}<\infty.
\end{equation}
Both $K_1$ and $K_2$ depend only on $n,A,|c_k|,C_S$ and $\beta$.
 It is easy to see that
\begin{equation}
    \begin{split}
        \frac{1}{K_2}&\le \int_{\{\rho< 2K_2 V_{\omega}\}} \rho^{-1}\omega^{n} + \int_{\{\rho \ge 2K_2 V_{\omega}\}}\rho^{-1} \omega^{n} \le \int_{\{\rho< 2K_2 V_{\omega}\}} \omega^{n} + \frac{1}{2K_2},
    \end{split}
\end{equation}
where $V_\omega=\int_{X}\omega^{n}$. Set $E=\{\rho< 2K_2 V_{\omega}\}, V_\omega(E)=\int_{E} \omega^{n}$. Then it follows that 
\begin{equation}\label{volu}
    \frac{1}{2 K_2} \le V_\omega(E) .
\end{equation}
 
 For $0<p<1$, equation \eqref{iden} gives
\begin{equation}\label{sm}
\int_M|\partial\rho^{p/2}|^2\omega^{n}\le C(p)\int_M\rho^{p}\omega^{n},\qquad C(p)=\frac n4\left(Ap+\frac{Ap+|c_k|p^2}{1-p}\right).
\end{equation}
In particular, $C(p)\to0$ as $p\to 0$. Combining \eqref{sm}   with the Poincare inequality \eqref{Poin}, we obtain
\begin{equation}\label{9.1}
    \begin{split}
        \int_{X} |\rho^{\frac{p}{2}}-\overline{{\rho^{\frac{p}{2}}}}|^2 \omega^n &\le C_P \int_{X}|\partial\rho^{\frac{p}{2}}|^2\omega^n\le C_P C(p) \int_X\rho^{p}\omega^{n} 
    \end{split}
\end{equation}
where $\overline{{\rho^{\frac{p}{2}}}}=\frac{1}{V_\omega}\int_X \rho^{\frac{p}{2}} \omega^n$. 

Integrating $\left(\overline {\rho^{\frac{p}{2}}}\right)^2\le2|\rho^{\frac{p}{2}}-\overline {\rho^{\frac{p}{2}}}|^2+2 \rho^{p}$ over $E$ and using \eqref{volu} gives
\begin{equation}\label{1.70}
    \begin{split}
         \left(\overline {\rho^{\frac{p}{2}}}\right)^2&\le \frac{2}{V_\omega(E)} \left(\int_E |\rho^{\frac{p}{2}}-\overline {\rho^{\frac{p}{2}}}|^2 \omega^n + \int_E (2K_2 V_\omega)^p \omega^n\right)\\&
        \le 4K_2 \int_E |\rho^{\frac{p}{2}}-\overline {\rho^{\frac{p}{2}}}|^2 \omega^n + 2 (2K_2 V_\omega)^p
    \end{split}
\end{equation}
It follows that 
\begin{equation}\label{1.71}
    \begin{split}
        \int_X
\rho^p\omega^n &= \int_E |\rho^{\frac{p}{2}}-\overline {\rho^{\frac{p}{2}}}|^2 \omega^n + V_\omega \left(\overline {\rho^{\frac{p}{2}}}\right)^2 \\&\le (4V_\omega K_2 +1)\int_E |\rho^{\frac{p}{2}}-\overline {\rho^{\frac{p}{2}}}|^2 \omega^n + 2 V_\omega (2K_2 V_\omega)^p \\& \le  C_P C(p)(4V_\omega K_2 +1)\int_X \rho^p \omega^n + 2 V_\omega(2K_2 V_\omega)^p
\end{split}
\end{equation}
The first inequality comes from \eqref{1.70}, the second inequality comes from \eqref{9.1}. Choose a fixed $p_0 \in (0,1)$ such that
\begin{equation}
    C_P C(p_0)(4V_\omega K_2 +1) \le \frac{1}{2}
\end{equation}
This is possible because $C(p)\to0$ as $p\to 0$. Let $p=p_0$, \eqref{1.71}  yields
\begin{equation}\label{1.75}
    \int_X
\rho^{p_0}\omega^n \le 4 V_\omega(2K_2 V_\omega)^{p_0}
\end{equation}
Combining \eqref{1.75}   with \eqref{1.73} , we obtain
\begin{equation}
    \begin{split}
        \|\rho\|^{2}_{L^\infty}&\le K_1^{2}\int_{X}\rho^2\omega^n \\&= K_1^{2}\int_{X}\rho^{p_0}\rho^{2-p_0}\omega^n\\& \le 
        K_1^{2}\|\rho\|^{2-p_0}_{L^\infty}\cdot 4 V_\omega(2K_2 V_\omega)^{p_0}
    \end{split}
\end{equation}
Therefore
\begin{equation}
    \|\rho\|_{L^\infty}\le K_1^{\frac{2}{p_0}}(4 V_\omega)^{\frac{1}{p_0}}\cdot2K_2 V_\omega
\end{equation}
It is easy to see the constant in the right of the above equation depends on $n,A,C_S$, $ C_P,V_\omega,|c_k|$ and $\beta$.

\end{proof}

\subsection{Interior estimates}\label{interior}
For $0\le k\le n-1$, define
\begin{equation}\label{eq:linear-operator}
\mathcal L_{k,\omega}u=\frac{n\ii\partial\bar\partial(u\omega^k)\wedge\omega^{n-k-1}}{\omega^n}.
\end{equation}
The product rule gives
\begin{equation}\label{110}
\mathcal L_{k,\omega}u=\Delta u+\langle B_{k,\omega},du\rangle+\varphi_{k,\omega}u,
\end{equation}
where
\[
\varphi_{k,\omega}=\frac{n\ii\partial\bar\partial\omega^k\wedge\omega^{n-k-1}}{\omega^n}
\]
and the smooth real $1$-form $B_{k,\omega}$ is determined by
\[
\langle B_{k,\omega}, \mathrm{d}u\rangle\omega^n=n\ii\bigl(\partial u\wedge\bar\partial\omega^k-\bar\partial u\wedge\partial\omega^k\bigr)\wedge\omega^{n-k-1}.
\]
For $k=0$, one has $B_{0,\omega}=0$, $\varphi_{0,\omega}=0$, and $\mathcal L_{0,\omega}=\Delta$. Thus $\mathcal L_{k,\omega}$ is a  second-order elliptic operator with principal part $\Delta$. Its coefficients depend smoothly on the metric and its derivatives. Equation \eqref{SGG} can be written as
\[ L_{k,\omega}\rho=nc_k\rho.\]
We record the local regularity statement needed below. All norms on relatively compact subsets in this statement are taken with respect to fixed coordinate charts and a fixed smooth background metric.

\begin{theorem}\label{grade}
Let $U$ be an open subset of a complex manifold of dimension $n$, and let $\omega_t$ be smooth Hermitian metrics converging to a Hermitian metric $\omega_0$ in $C^\infty_{\mathrm{loc}}(U)$ as $t\to0$. Suppose that $|c_t|\le A_0$ and that smooth functions $\rho_t$ satisfy
\[
\mathcal L_{k,\omega_t}\rho_t=nc_t\rho_t\quad\text{on }U,\qquad \|\rho_t\|_{L^\infty(U)}\le M_0.
\]
For every $U'\Subset U$, integer $m\ge0$, and $0<\gamma<1$, there is a constant $C$ such that
\begin{equation}\label{eq:interior-regularity}
\|\rho_t\|_{C^{m,\gamma}(U')}\le CM_0
\end{equation}
for all sufficiently small $t$. The constant depends on $U',U,m,\gamma,n,k,A_0$ and the local ellipticity and coefficient bounds. If $\rho_t\ge a>0$ on $U$, then $\log\rho_t$ also has uniform local $C^{m,\gamma}$-bounds, with constants additionally depending on $a$ and $M_0$.
\end{theorem}

\begin{proof}
Choose $U'\Subset U_1\Subset U_2\Subset U$. The operators $\mathcal L_{k,\omega_t}-nc_t$ are uniformly elliptic on $U_2$. Their coefficients have uniform bounds  on $U_2$, by \eqref{110} and the convergence of $\omega_t$. Interior $W^{2,p}$-estimates for linear elliptic equations give
\[
\|\rho_t\|_{W^{2,p}(U_1)}\le C_p\|\rho_t\|_{L^p(U_2)}\le C_p'M_0,\qquad 1<p<\infty;
\]
see \cite[Chapters 6 and 9]{GT}. Choose $p>2n$ so that $1-2n/p>\gamma$. Sobolev embedding gives a uniform $C^{1,\gamma}$-bound on a smaller domain. Interior Schauder estimates and iteration yield \eqref{eq:interior-regularity}. The estimates for $\log\rho_t$ follow by the chain rule and the lower bound $\rho_t\ge a$.
\end{proof}

\section{Existence of singular generalized Gauduchon metrics}\label{ext}

For an irreducible reduced complex space $X$ endowed with a Hermitian metric $\omega$, there exists a resolution $\pi: \widetilde{X} \to X$ of the singularities of $X$. The map $\pi$ is an isomorphism away from the singular locus of $X$. Denote by $D$ the exceptional divisor in $\widetilde{X}$. One can then choose a smooth Hermitian metric $h$ on the line bundle $L = \mathcal{O}_{{\widetilde{X}}}(-D)$ and a constant $K > 0$ such that the real $(1,1)$-form $\theta = \pi^*\omega + K c_1(L, h)$ defines a Hermitian metric on $\widetilde{X}$ \cite[Lemma 2.2]{CMM}. Set 
\[\alpha=\pi^*\omega,\qquad \kappa = K c_1(L, h). \]
For $\epsilon \in (0,1]$, consider a family of Hermitian metrics defined by 
\[\omega_{\epsilon} = \alpha + \epsilon \theta= (1+\epsilon)\alpha + \epsilon\kappa \]
These are Hermitian metrics on $\widetilde X$, and $\omega_\epsilon\to\alpha$ smoothly on $\widetilde X$. The limit $\alpha$ is positive on $\widetilde X\setminus D$. Since $d\kappa=0$,
\begin{equation}\label{eq:derivative-relations}
\partial\omega_\epsilon=(1+\epsilon)\partial\alpha,\qquad \partial\bar\partial\omega_\epsilon=(1+\epsilon)\partial\bar\partial\alpha.
\end{equation}

By the theorem of Fu-Wang-Wang \cite{FWW}, for each $\epsilon$, there exists a unique smooth positive function $\rho_\epsilon$ satisfying the equation
\begin{equation}\label{f_k-Gau}
    \sqrt{-1}\partial\bar{\partial}(\rho_{\epsilon} \omega^{k}_{\epsilon})\wedge\omega_{\epsilon}^{n-k-1}=c_{k,\epsilon}\rho_{\epsilon}\omega_{\epsilon}^{n},\;\;\inf\rho_{\epsilon}=1.
\end{equation} 
By Theorem  \ref{UniSo}, the Sobolev constant and Poincar\'e constant are uniformly bounded for $\epsilon\in(0,1]$.

 The constants $A_\epsilon$ appearing in inequalities \eqref{1.5} and \eqref{1.7} for the metrics $\omega_\epsilon$ are independent of $\epsilon$.

\begin{lemma}\label{uniform2}
    Let $\rho_\epsilon$ be the unique solution and $c_{k,\epsilon}$ be the unique constant satisfying \eqref{f_k-Gau}, then there exists a uniform constant $A$, independent of $\epsilon\in(0,1]$, such that the following inequalities
    \begin{equation}\label{1.501}
    -A\omega_{\epsilon}^{n}\le \sqrt{-1}\partial\bar{\partial}\omega_{\epsilon}^{k}\wedge\omega_{\epsilon}^{n-k-1}\le A\omega_{\epsilon}^{n}
\end{equation}
\begin{equation}\label{1.701}
    -A\omega_{\epsilon}^{n}\le \sqrt{-1}\partial\bar{\partial}\omega_{\epsilon}^{n-k-1}\wedge\omega_{\epsilon}^{k}\le A\omega_{\epsilon}^{n}
\end{equation}
   hold for all $\epsilon\in(0,1]$.   
\end{lemma}

\begin{proof}
Set $\alpha=\pi^*\omega$.
Choose finitely many relatively compact local embeddings covering $X$, and let $\Omega$ denote a smooth ambient Hermitian extension of $\omega$ on one such chart. Its derivatives are bounded on a neighborhood of the relevant compact set. On the inverse image of this chart, write $\Phi=j\circ\pi$. Then $\alpha=\Phi^*\Omega$ and $\alpha\le\omega_\epsilon$. Thus the norm of $d\Phi$ from $(T\widetilde X,\omega_\epsilon)$ to the ambient Hermitian tangent space is at most one. Pulling back the bounded ambient tensors gives
\[
|\ii\partial\bar\partial\alpha|_{\omega_\epsilon}\le C,\qquad |\partial\alpha|_{\omega_\epsilon}\le C
\]
with a uniform constant. These inequalities remain valid on $D$, where $d\Phi$ need not be injective. The finite cover gives one constant for all of $\widetilde X$.

For $k=0$, the middle form in \eqref{1.501} is zero. For $k=1$, it equals
\[
(1+\epsilon)\ii\partial\bar\partial\alpha\wedge\omega_\epsilon^{n-2}.
\]
For $2\le k\le n-1$, the product rule and \eqref{1.501} give
\[
\begin{aligned}
\ii\partial\bar\partial\omega_\epsilon^k\wedge\omega_\epsilon^{n-k-1}={}&k(1+\epsilon)\ii\partial\bar\partial\alpha\wedge\omega_\epsilon^{n-2}\\
&+k(k-1)(1+\epsilon)^2\ii\partial\alpha\wedge\bar\partial\alpha\wedge\omega_\epsilon^{n-3}.
\end{aligned}
\]
The last formula is used only when $k\ge2$, so $n\ge3$. The tensor bounds, together with $1+\epsilon\le2$, bound the absolute value of the coefficient of each top-degree form relative to $\omega_\epsilon^n$. This proves \eqref{1.501}, including the case $n=2$.
\end{proof}

Under the condition: \eqref{1.501} and \eqref{1.701}, we proved  that $|c_{k,\epsilon}|$ is uniformly bounded for all $\epsilon \in (0,1]$.

\begin{lemma}\label{uniform}
    If the following inequalities hold for all $\epsilon\in(0,1]$
\begin{equation}\label{4.8}
-A\omega^{n}_{\epsilon} \le \sqrt{-1}\partial\bar{\partial}\omega^{n-k-1}_{\epsilon} \wedge \omega^{k}_{\epsilon} \le A\omega^{n}_{\epsilon},
\end{equation}
then
\begin{equation}\label{cst}
    |c_{k,\epsilon}|\le A
\end{equation}
     
\end{lemma}
\begin{proof}
     The constant $c_{k,\epsilon}$ is given by
\begin{equation}\label{4.10}
    \begin{split}
        c_{k,\epsilon}&=\frac{\int_{X}\sqrt{-1}\partial\bar{\partial}(\rho_{\epsilon}\omega^{k}_{\epsilon})\wedge\omega^{n-k-1}_{\epsilon}}{\int_{X}\rho_{\epsilon}\omega^{n}_{\epsilon}}\\&
=\frac{\int_{X}\rho_{\epsilon}\sqrt{-1}\partial\bar{\partial}\omega^{n-k-1}_{\epsilon}\wedge\omega^{k}_{\epsilon}}{\int_{X}\rho_{\epsilon}\omega^{n}_{\epsilon}}
    \end{split}
\end{equation}
Using \eqref{4.8} and \eqref{4.10}, we  obtain
\begin{equation}\label{4.15}
    \begin{split}
        |c_{k,\epsilon}|\le A.
    \end{split}
\end{equation}
for $\epsilon\in(0,1]$.
\end{proof}

\begin{proof}[Proof of Theorem \ref{MT} ]
We first prove the existence of a solution to equation \eqref{SGG}. On every relatively compact open subset of $\widetilde X\setminus D$, the metrics $\omega_\epsilon$ converge smoothly to the Hermitian metric $\pi^*\omega$. 
The volumes satisfy
\begin{equation}\label{eq:volume-bounds}
0<V_0:=\int_{\widetilde X}(\pi^{*}\omega)^n=\int_{X_{\reg}}\omega^n\le\int_{\widetilde X}\omega_\epsilon^n\le\int_{\widetilde X}\omega_1^n=:V_1.
\end{equation}
Fix $1<\beta<n/(n-1)$. By Theorem~\ref{UniSo}, the Sobolev and Poincar\'e constants are uniformly bounded. The constant $c_{k,\epsilon}$ is uniformly bounded by Lemma \ref{uniform}. Applying Theorem \ref{C0}, we obtain   the uniform $L^{\infty}$ estimate 
\begin{equation}\label{eq:uniform-rho-bound}
1\le\rho_\epsilon\le C
\end{equation}
with a constant independent of $\epsilon$.  Apply Theorem~\ref{grade} to equation \eqref{par_k_gau}, we have the interior higher order estimates.

 Using a diagonal argument, we infer that there is a subsequence $\rho_{\epsilon}$ convergencing in $C^{\infty}_{\mathrm{loc}}(\widetilde{X}\backslash D)$ to a function $\widetilde\rho\in C^{\infty}_{\mathrm{loc}}(\widetilde{X}\backslash D)\cap L^{\infty}_{\mathrm{loc}}(\widetilde{X}\backslash D)$ which is also bounded away from zero and a subsequence $c_{k,\epsilon}\to c_k$. Take $\epsilon\to 0$ in \eqref{f_k-Gau}, we obtain
\begin{equation}\label{s_k-Gau}
    \sqrt{-1}\partial\bar{\partial}(\widetilde\rho (\pi^{*}\omega)^{k})\wedge\omega^{n-k-1}=c_{k}\widetilde\rho(\pi^{*}\omega)^{n},\quad\text{on }\widetilde X\setminus D.
\end{equation}
The isomorphism $\pi:\widetilde X\setminus D\to X_{\reg}$ transfers $\widetilde\rho$ to a function $\rho$ satisfying \eqref{SGG} and $1\le\rho\le C$. The normalization of the approximating functions need not pass to their infimum on the regular locus. If the normalization $\inf_{X_{\reg}}\rho=1$ is desired, divide the limit by its positive infimum. This does not change the equation or $c_k$.


We recall the existence of cut-off function on compact complex varieties \cite[Lemma 3.2]{Pan1}
\begin{lemma}\label{cut}
    Let $X$ be a compact irreducible reduced complex space and $\omega$ be a Hermitian metric on $X$. Then there exist cut-off functions $\eta_{\epsilon}\in C_{c}^{\infty}(X_{\mathrm{reg}})$ for  any $\epsilon >0$ satisfying the following properties 
    \begin{enumerate}
        \item The family $\eta_\epsilon$  increase to the characteristic function of $X_{\mathrm{reg}}$ when $\epsilon$ decreases to zero.
        \item $\lim_{\epsilon\to0}\int_{X}|\sqrt{-1}\partial\bar{\partial}\eta_{\epsilon}\wedge\omega^{n-1}|=0$
          \item $\lim_{\epsilon\to0}\int_{X}\sqrt{-1}\partial\eta_{\epsilon}\wedge\bar{\partial}\eta_{\epsilon}\wedge\omega^{n-1}=0$.
    \end{enumerate}
\end{lemma}

These functions enable us to carry out integration by parts even on compact complex varieties. For $1\le k \le n-2$, 
    Let $\rho$ be
    a positive  function  $\rho\in C^{\infty}(X_{\mathrm{reg}})\cap L^{\infty}(X_{\mathrm{reg}})$ , and let $c_k$ be a constant satisfying 
    \begin{equation}\label{1.42}
    \sqrt{-1}\partial\bar{\partial}(\rho\omega^{k})\wedge\omega^{n-k-1}=c_{k}\rho\omega^{n}
\end{equation}
on $X_{\mathrm{reg}}$. Let $\eta_\epsilon$ be the cut-off functions given in Lemma \ref{cut}. Integrating both sides of \eqref{1.42} yields
\begin{equation}\label{1.41}
    \begin{split}
        c_{k}=\frac{\int_{X}\eta_{\epsilon}\sqrt{-1}\partial\bar{\partial}(\rho\omega^{k})\wedge\omega^{n-k-1}}{\int_{X}\rho\eta_{\epsilon}\omega^{n}}
    \end{split}
\end{equation}
For $k=n-1$, using \eqref{1.41}, By Stokes' formula and direct computations, we know that 
\begin{equation}
    \begin{split}
        c_{n-1}&=\frac{\int_{X}\eta_{\epsilon}\sqrt{-1}\partial\bar{\partial}(\rho\omega^{n-1})}{\int_{X}\rho\eta_{\epsilon}\omega^{n}}=\frac{\int_{X}\rho\sqrt{-1}\partial\bar{\partial}\eta_{\epsilon}\wedge\omega^{n-1}}{\int_{X}\rho\eta_{\epsilon}\omega^{n}}
    \end{split}
\end{equation}
As $\rho$ is bounded, Lemma \ref{cut}  implies
\begin{equation}
    \lim_{\epsilon\to 0} \int_{X}\rho\sqrt{-1}\partial\bar{\partial}\eta_{\epsilon}\wedge\omega^{n-1}=0.
\end{equation}
Hence, $c_{n-1}=0$. If $\omega$ is a K\"ahler metric, i.e., $\mathrm{d}\omega=0$ on $X_{\mathrm{reg}}$, then from \eqref{1.41}, the closedness of $\omega$, and Stokes' theorem, we conclude that
\begin{equation}
    c_k= 0.
\end{equation}
for $1\le k \le n-1$. Mutiple $\rho\eta_\epsilon$ both sides of \eqref{1.42} and Integrating both sides of \eqref{1.42} yeilds
\begin{equation}
    \begin{split}
0&=\int_{X}\rho\eta_\epsilon\sqrt{-1}\partial\bar{\partial}(\rho\omega^{k})\wedge\omega^{n-k-1}\\&=-\int_{X}\eta_\epsilon\sqrt{-1}\partial\rho\wedge\bar{\partial}\rho\wedge\omega^{n-1}-\frac{1}{2}\int_{X}\sqrt{-1}\partial\eta_\epsilon\wedge\bar{\partial}\rho^{2}\wedge\omega^{n-1}\\&
=-\int_{X}\eta_\epsilon\sqrt{-1}\partial\rho\wedge\bar{\partial}\rho\wedge\omega^{n-1}+\frac{1}{2}\int_{X}\rho^{2}\sqrt{-1}\partial\bar{\partial}\eta_\epsilon\wedge\omega^{n-1}\\&
=-n\int_{X}\eta_\epsilon|\partial\rho|^{2}\omega^{n}+\frac{1}{2}\int_{X}\rho^{2}\sqrt{-1}\partial\bar{\partial}\eta_\epsilon\wedge\omega^{n-1}
    \end{split}
\end{equation}
By Lemma \ref{cut}, the second term on the right-hand side of the last equality converges to zero as $\epsilon$ tends to zero. Hence, we have
\begin{equation}
    \int_{X}|\mathrm{d}\rho|^{2}\omega^{n}=0.
\end{equation}
Since $X_{\mathrm{reg}}$ is connected by the irreducible assumption, $\rho$ is a positive constant.
\end{proof}

 Finally, we will give explicit singular non-K\"ahler examples admitting bounded $1$-Gauduchon metrics.

\begin{example}[A singular non-K\"ahler threefold]\label{ex:threefold}
Let
\[
H=(\C^2\setminus\{0\})/\langle(z_1,z_2)\mapsto(2z_1,2z_2)\rangle
\]
be the standard Hopf surface. It is diffeomorphic to $S^3\times S^1$, so $b_2(H)=0$. In particular, it is not K\"ahler. Indeed, a K\"ahler form on $H$ would be exact and would have strictly positive square integral, contradicting Stokes' theorem. Choose a Gauduchon metric $\omega_H$ on $H$. Since $\dim_\C H=2$, it satisfies $\partial\bar\partial\omega_H=0$.

Let $C\subset\mathbb P^2$ be the irreducible nodal cubic
\[
y^2z=x^2(x+z),
\]
and let $p=[0:0:1]$ be its node. The affine equation at $p$ is $y^2=x^2(1+x)$, whose quadratic part has two distinct linear factors. The other points are smooth. Let $\omega_C$ be the restriction of the Fubini--Study form to $C$. Then
\[
X=H\times C,\qquad \omega=\operatorname{pr}_H^*\omega_H+\operatorname{pr}_C^*\omega_C
\]
defines an irreducible reduced compact Hermitian threefold. Its singular locus is $H\times\{p\}$. The metric is the restriction of the smooth Hermitian form $\operatorname{pr}_H^*\omega_H+\operatorname{pr}_{\mathbb P^2}^*\omega_{\mathrm{FS}}$ on $H\times\mathbb P^2$. On $X_{\reg}$,
\[
\partial\bar\partial\omega=0.
\]
Thus $\omega$ is a bounded $1$-Gauduchon metric, with conformal factor $\rho=1$. The space $X$ is non-K\"ahler, since the restriction of a K\"ahler metric to $H\times\{q\}$, for $q\in C_{\reg}$, would be a K\"ahler metric on $H$. This example is not normal, and its singularities are not isolated.
\end{example}

\begin{example}[A normal example with quotient singularities]\label{ex:quotient}
Let
\[
Q=\{[x_0:x_1:x_2:x_3]\in\mathbb P^3:x_0x_1=x_2^2\}
\]
be the quadric cone. Its only singular point is $v=[0:0:0:1]$. The affine germ at $v$ is $ab=c^2$, which is the quotient of $\C^2$ by $(u,w)\mapsto(-u,-w)$, since
\[
\C[u,w]^{\{\pm1\}}=\C[u^2,w^2,uw]\cong\C[a,b,c]/(ab-c^2).
\]
In particular, $Q$ is normal. Let $\omega_Q$ be the restriction of the Fubini--Study metric. Using the Hopf surface and metric from Example~\ref{ex:threefold}, set
\[
X=H\times Q,\qquad \omega=\operatorname{pr}_H^*\omega_H+\operatorname{pr}_Q^*\omega_Q.
\]
Then $X$ is a normal compact non-K\"ahler fourfold with singular locus $H\times\{v\}$. Its local singular model is $\C^2\times(\C^2/\{\pm1\})$. Again $\partial\bar\partial\omega=0$ on $X_{\reg}$, so $\omega$ is a bounded $1$-Gauduchon metric. Here $\omega$ is the restriction of a smooth Hermitian form on $H\times\mathbb P^3$.
\end{example}

\section{Uniqueness}\label{uniq}

In this section, we will prove the uniqueness of the constant $c_k$ and $\rho$  up to multiplication by a positive constant in the equation \eqref{SGG}. The generalized Gauduchon equation \eqref{SGG} can be written as 
\begin{equation}\label{par_k_gau}
    \begin{split}
        (\mathcal L_{k,\omega}-nc_{k})\rho_\epsilon=0.
    \end{split}
\end{equation}
We write $\mathcal L_k=\mathcal L_{k,\omega}$, $B_k=B_{k,\omega}$, and $\varphi_k=\varphi_{k,\omega}$ for $1\le k\le n-1$. All norms are taken with respect to the given Hermitian metric $\omega$ on a compact complex space $X$. Integrations on a compact complex space are taken over its regular part, where the measure is $\omega^n$.

\begin{lemma}\label{lem:formal-adjoint}
For every $1\le k\le n-1$, the coefficients $B_k$ and $\varphi_k$ are bounded on $X_\reg$ with respect to $\omega$, we have
\[ \int_X g\mathcal L_k f\,\omega^n=\int_X f\mathcal L_{n-k-1}g\,\omega^n\]
whenever $f,g$ are smooth  functions and at least one of them has compact support in $X_\reg$. Thus the formal adjoint of $\mathcal L_k$  is $\mathcal L_{n-k-1}$, i. e., $\mathcal L_k^*=\mathcal L_{n-k-1}$.
\end{lemma}

\begin{proof}
By the definition of complex spaces and compactness, we can choose a finite cover $\{U_\alpha\}_{\alpha\in I}$ of $X$ by local embeddings $\tau_\alpha:U_\alpha \to V_\alpha\subset \mathbb{C}^{N}$. By means of local embeddings $\tau_\alpha$, $\omega$ can be seen as restriction of a smooth Hermitian metric $\Omega_\alpha$ on the ambient space $V_\alpha$. On the regular locus, the embedding is isometric from $\omega$ to $\Omega_\alpha$, and exterior differentiation commutes with restriction. The restrictions of $\partial\Omega_\alpha^k$, $\bar\partial\Omega_\alpha^k$, and $\partial\bar\partial\Omega_\alpha^k$ therefore have uniformly bounded norms on the part of the chart under consideration. Their contractions with $\omega$, as in \eqref{110}, give the required bounds for $B_k$ and $\varphi_k$.

If one of $f,g$ has compact support, Stokes' theorem gives
\[
\begin{aligned}
\int_X g\mathcal L_k f\,\omega^n
&=n\int_X g\omega^{n-k-1}\wedge\ii\partial\bar\partial(f\omega^k)\\
&=n\int_X f\omega^k\wedge\ii\partial\bar\partial(g\omega^{n-k-1})\\
&=\int_X f\mathcal L_{n-k-1}g\,\omega^n.
\end{aligned}
\]
\end{proof}

 Let $\eta_m$ be the sequence from Lemma~\ref{cut}. 
It follows that
\begin{equation}\label{eq:cutoff-norms}
\|\Delta\eta_m\|_{L^1}\longrightarrow0,\qquad \|\mathrm{d}\eta_m\|_{L^2}\longrightarrow0.
\end{equation}
In the following estimates, all norms are taken with respect to $\omega$ and $\omega^n$.

\begin{lemma}\label{cutoff1}
Fix $1\le k\le n-1$, and set $\ell=n-k-1$. Suppose that $v_k\in C^\infty(X_{\reg})\cap L^\infty(X_{\reg})$ and satisfies $\mathcal L_k v_k=\lambda_k v_k$ for some $\lambda_k\in\R$. Then  
\[ \mathrm{d}v_k\in L^2(X_{\reg}),\qquad\|(\mathcal L_k-\lambda_k)(\eta_m v_k)\|_{L^1}\to0,\qquad m\to+\infty.  \]
\end{lemma}

\begin{proof}
By the product rule,
\[
\mathcal L_k(v_k^2)=2v_k\mathcal L_k v_k-\varphi_k v_k^2+2|\partial v_k|^2=(2\lambda_k-\varphi_k)v_k^2+2|\partial v_k|^2.
\]
Multiply this identity by $\eta_m$ and use Lemma~\ref{lem:formal-adjoint}. Since $\eta_m$ has compact support,
\[
\begin{aligned}
2\int_X\eta_m|\partial v_k|^2\,\omega^n
&=\int_Xv_k^2\mathcal L_\ell\eta_m\,\omega^{n}-\int_X\eta_m(2\lambda_k-\varphi_k)v_k^2\,\omega^n\\
&=\int_Xv_k^2\Delta\eta_m\,\omega^n+\int_Xv_k^2\langle B_\ell,\mathrm{d}\eta_m\rangle\,\omega^n\\
&\quad+\int_X\eta_m(\varphi_k+\varphi_\ell-2\lambda_k)v_k^2\,\omega^n.
\end{aligned}
\]
The right-hand side is uniformly bounded. Indeed, $v_k$, $B_k$, $\varphi_k$, and $\varphi_\ell$ are bounded, and \eqref{eq:cutoff-norms} controls the first two terms. Fatou's lemma therefore gives $\int_X|\partial v_k|^2\,\omega^n<\infty$. By \eqref{eq:norm-convention}, this is equivalent to $\mathrm{d}v_k\in L^2(X)$.

Using $\mathcal L_k v_k=\lambda_kv_k$ once more, we obtain
\[
(\mathcal L_k-\lambda_k)(\eta_m v_k)=v_k\Delta\eta_m+\langle \mathrm{d}\eta_m,\mathrm{d}v_k\rangle+v_k\langle B_k,\mathrm{d}\eta_m\rangle.
\]
Consequently,
\[
\begin{aligned}
\|(\mathcal L_k-\lambda_k)(\eta_m v_k)\|_{L^1}
&\le\|v_k\|_{L^\infty}\|\Delta\eta_m\|_{L^1}+\|\mathrm{d}\eta_m\|_{L^2}\|\mathrm{d}v_k\|_{L^2}\\
&\quad+\|v_k\|_{L^\infty}\|B_k\|_{L^\infty}V_{\omega}^{1/2}\|\mathrm{d}\eta_m\|_{L^2}
\longrightarrow0,\qquad m\to +\infty.
\end{aligned}
\]
\end{proof}

For such a function $v_k$ and any $F\in C^\infty(X)\cap L^\infty(X)$, compactly supported integration by parts gives
\begin{equation}\label{cutoff}
\int_X\eta_m v_k(\mathcal L_\ell-\lambda_k)(F)\,\omega^n=\int_X F (\mathcal L_k-\lambda_k)(\eta_m v_k)\omega^n\longrightarrow0.
\end{equation}
Each integral on the left is over a compact subset of $X$. Thus this identity does not require any global integrability assumption on the derivatives of $F$.

\begin{lemma}\label{200}
    Fix $1\le k\le n-1$, and set $\ell=n-k-1$.  Suppose that $v_k, v_\ell\in C^\infty(X_{\reg})\cap L^\infty(X_{\reg})$, $v_k, v_\ell >0$ and satisfies 
    \[\mathcal L_k v_k=\lambda_k v_k,\qquad \mathcal L_\ell v_\ell=\lambda_\ell v_\ell \]
    for some $\lambda_k,\lambda_\ell\in\R$. Then
    \[ \lambda_k=\lambda_\ell.\]
\end{lemma}
\begin{proof}
Choose $F=v_\ell$ in the equation \eqref{cutoff}, we have
    \begin{equation}
        \begin{split}
         (\lambda_\ell - \lambda_k) \int_X\eta_m v_kv_\ell\,\omega^n&= \int_X\eta_m v_k(\mathcal L_\ell-\lambda_k)(v_\ell)\,\omega^n\\&=\int_X v_\ell (\mathcal L_k-\lambda_k)(\eta_m v_k)\omega^n \to 0.
        \end{split}
    \end{equation}
    By dominated convergence, $\int_X\eta_m v_kv_\ell\,\omega^n\to \int_X v_kv_\ell\,\omega^n> 0$.
Hence, it follows that
\begin{equation}
    \lambda_\ell =\lambda_k
\end{equation}
\end{proof}

\begin{proposition}\label{uni}
    Under the assumptions of Theorem~\ref{MT}, let $u_1,u_2\in C^\infty(X_{\reg})\cap L^\infty(X_{\reg})$ be positive and suppose that
\begin{equation}\label{SGG1}
\ii\partial\bar\partial(u_i\omega^k)\wedge\omega^{n-k-1}=a_iu_i\omega^n,\qquad i=1,2,\qquad 1\le k\le n-1
\end{equation}
for constants $a_1,a_2$. Then $a_1=a_2$ and $u_2=A u_1$ for some $A>0$.
\end{proposition}
\begin{proof}
 Equation \eqref{SGG1} can be written as
\begin{equation}\label{SGG2}
\mathcal L_{k}u_i=na_i u_i,\qquad i=1,2,
\end{equation}


For $1\le j \le n-1$, the  existence argument in Section \eqref{ext} gives  bounded smooth positive functions $v_j$ satisfying 
\begin{equation}\label{f_k-Gau}
   \mathcal L_j v_j=\lambda_j v_j,\;\;\;\;\inf v_{j}=1
\end{equation} 
 for some $\lam_j\in\R$.  Set $\ell=n-k-1$, if  $k=n-1$, then $\ell=0$. Let $\mathcal L_0=\Delta$,  then $v_0=1$ and $\lam_0=0$ satisfy the above equation. By Lemma \ref{200}, we have
     \[ \lam_\ell=na_1=na_2 \]
Hence, $a_1=a_2$. This proves uniqueness of the constant. 

To compare the functions, set $r=u_1/v_k$. Hence, $v_kr=u_1$. Since $\lam_k=\lam_\ell$,  both $u_1$ and $v_k$ solve the equation with eigenvalue $\lam_l$, the product rule gives
\[
0=(\mathcal L_k-\lam_\ell)(v_kr)=v_k\Delta r+v_k\langle B_k,\mathrm{d}r\rangle+\langle \mathrm{d}v_k,\mathrm{d}r\rangle.
\]
Expanding $(\mathcal L_k-\lam_\ell)(v_kr^2)$ and using this equality yields
\[
(\mathcal L_k-\lam_\ell)(v_kr^2)=2r\bigl(v_k\Delta r+v_k\langle B_k,\mathrm{d}r\rangle+\langle \mathrm{d}v_k,\mathrm{d}r\rangle\bigr)+2v_k|\partial r|^2.
\]
Therefore
\begin{equation}\label{eq:ratio-energy}
(\mathcal L_k-\lam_\ell)\left(\frac{u_1^2}{v_k}\right)=2v_k\left|\partial\left(\frac{u_1}{v_k}\right)\right|^2.
\end{equation}
Apply \eqref{cutoff} to $F=u_1^2/v_k$. This gives
\begin{equation}
    \begin{split}
        2\int_X\eta_m v_kv_\ell\left|\partial\left(\frac {u_1}{v_k}\right)\right|^2\,\omega^n&=\int_X\eta_m v_\ell(\mathcal L_k-\lam_\ell)\left(\frac{u_1^2}{v_k}\right)\omega^n\\&
        = \int_X \frac{u_1^2}{v_k}(\mathcal L_l-\lam_\ell)(\eta_m v_\ell)\omega^n
    \end{split}
\end{equation}
Since the infimum of $v_k$ is 1 and $u_1$ is bounded, then $\frac{u_1^2}{v_k}$ is bounded. The integrands on the left are nonnegative. By Lemma \ref{cutoff1} and Fatou's lemma, we have
\[
\int_X v_kv_\ell\left|\partial\left(\frac {u_1}{v_k}\right)\right|^2\,\omega^n=0.
\]
Since $v_k,v_\ell>0$, we have $\mathrm{d}\left(\frac {u_1}{v_k}\right)=0$. The regular locus of an irreducible complex space is connected, so $u_1=A_1 v_k$ for some $A_1>0$. The same argument above gives $u_2=A_2 v_k$ for some $A_2>0$. Set $A=\frac{A_1}{A_2}$. Thus   $u_1=A u_2$.
\end{proof}

Proposition~\ref{uni} completes the proof of Theorem~\ref{MT}. It also shows that every bounded smooth positive solution is bounded away from zero. Either normalization
\[
\inf_{X_{\mathrm{reg}}}\rho=1\qquad\text{or}\qquad\int_{X_{\mathrm{reg}}}\rho\omega^n=1
\]
determines a unique solution. For $k=n-1$, the uniqueness of bounded Gauduchon metrics was proved in \cite[Theorem B]{Pan1}.

By Lemma \eqref{200}, we observe that 
\begin{corollary}\label{cor:adjoint-symmetry}
Let $(X,\omega)$ satisfy the assumptions of Theorem~\ref{MT}, then
\begin{equation}\label{eq:adjoint-symmetry}
c_k(\omega)=c_{n-k-1}(\omega),\qquad 1\le k\le n-2.
\end{equation}
\end{corollary}

\begin{remark}
    All the arguments in this section also hold on compact complex manifolds.
\end{remark}

\section{Proof of Theorem~\ref{thm2}}\label{blowups}

We first establish the uniform inequalities for blow-ups of smooth manifolds. We then prove the removable singularities statement and pass to the limit on the base.

The following extension lemma for solution of \eqref{ellip} is important for us to prove Theorem \ref{thm2}. 

\begin{proposition}\label{exten}
Let $M$ be a complex manifold of dimension $n\ge2$, let $\omega$ be a smooth Hermitian metric on $M$, and let $Z\subset M$ be a closed analytic subset with $\codim_\C Z\ge2$. Fix $1\le k\le n-1$. Suppose that $u\in C^\infty(M\setminus Z)\cap L^\infty_{\mathrm{loc}}(M)$ satisfies
\begin{equation}\label{ellip}
\ii\partial\bar\partial(u\omega^k)\wedge\omega^{n-k-1}=c_ku\omega^n\quad\text{on }M\setminus Z
\end{equation}
for a real constant $c_k$. Then $u$ extends uniquely to a smooth solution on $M$. Bounds $a\le u\le b$ on $M\setminus Z$ are preserved by the extension.
\end{proposition}

\begin{proof}
Set $L=\mathcal L_{k,\omega}-nc_k$. It is uniformly elliptic on relatively compact coordinate charts and has smooth coefficients. We first suppose that $Z$ is smooth near a given point, of complex codimension $r\ge2$. Choose coordinates $(z',z'')$ with $Z=\{z''=0\}$ and $z''\in\C^r$. Let $\chi_\delta$ be a smooth function equal to one when $|z''|\le\delta$ and zero when $|z''|\ge2\delta$, obtained by scaling a fixed cutoff. For $j=0,1,2$, its derivatives satisfy $|D^j\chi_\delta|\le C_j\delta^{-j}$.

Fix a compactly supported test function $\psi$ in this chart, and let $L^*$ be the formal adjoint with respect to coordinate Lebesgue measure. Since $(1-\chi_\delta)\psi$ is supported away from $Z$, the equation gives
\[
\int uL^*\psi=\int uL^*(\chi_\delta\psi).
\]
On the fixed support of $\psi$, the tube $\{|z''|<2\delta\}$ has volume $O(\delta^{2r})$. The coefficients of $L^*$ and the derivatives of $\psi$ are bounded. Hence
\[
\left|\int uL^*(\chi_\delta\psi)\right|\le C\|u\|_{L^\infty}\bigl(\delta^{2r}+\delta^{2r-1}+\delta^{2r-2}\bigr)\longrightarrow0.
\]
Thus $Lu=0$ across $Z$ in the sense of distributions. Elliptic regularity gives a smooth extension. Its upper and lower bounds are preserved by continuity.

For a general analytic set, apply this argument at every point of its regular locus. The resulting local extensions agree on overlaps and remain locally bounded. The remaining set is the singular locus of the reduced analytic set $Z$, which is analytic and has smaller dimension. Repeat the argument. After finitely many steps, no set remains. The extension is unique because $M\setminus Z$ is dense.
\end{proof}

Fu-Wang-Wang proved that 
$c_{k}(\omega)$ depends smoothly on  Hermitian metric $\omega$ (see \cite[Proposition 9]{FWW}). The following Proposition shows that this argument also holds for a degenerate Hermitian form arising from a blow-up of a compact complex manifold.

\begin{proposition}\label{prop:blowup-convergence}
Let $f:X\to Y$ be the blow-up of a compact  complex manifold along a smooth center of complex codimension at least two. For every Hermitian metric $\omega$ on $Y$, there is a Hermitian metric $\theta$ on $X$ such that $\theta-f^*\omega$ is closed and
\begin{equation}\label{eq:blowup-eigenvalue-limit}
\lim_{\epsilon\to0}c_k(f^*\omega+\epsilon\theta)=c_k(\omega),\qquad 1\le k\le n-2.
\end{equation}
\end{proposition}

\begin{proof}

Let $Z$ be the center and $E=f^{-1}(Z)$ the exceptional divisor. The map $f$ is an isomorphism away from the smooth center $Z$. One can then choose a smooth Hermitian metric $h$ on the line bundle $L = \mathcal{O}_{{X}}(-D)$ and a constant $K > 0$ such that the real $(1,1)$-form $\theta = \pi^*\omega + K c_1(L, h)$ defines a Hermitian metric on $X$ \cite[Lemma 2.2]{CMM}. For $\epsilon \in (0,1]$, consider the family of Hermitian metrics defined by $\omega_{\epsilon} = \pi^* \omega + \epsilon \theta$.

Let $c_{k,\epsilon}=c_k(\omega_\epsilon)$, and let $\rho_\epsilon$ be the corresponding solution normalized by $\min_X\rho_\epsilon=1$. The proof of Lemma~\ref{uniform2} applies without change: it uses only the ambient smoothness of the base metric, the inequality $f^*\omega\le\omega_\epsilon$, and the closed correction. Hence the top-degree form bounds are uniform. Equation \eqref{cst} gives $|c_{k,\epsilon}|\le A$. The volume has a positive lower bound and a finite upper bound. Following Pan' work, it is easy to see that uniform Sobolev inequalities and Poincare inequalities for the perturbation path of the pullback metric on the blow-up of Hermitian manifolds.  Therefore we have $1\le\rho_\epsilon\le C_0$ uniformly.

Take any sequence $\epsilon_j\to0$. By Theorem~\ref{grade}, there is a subsequence such that
\[
c_{k,\epsilon_j}\longrightarrow c_k',\qquad \rho_{\epsilon_j}\longrightarrow\rho\quad\text{in }C^\infty_{\mathrm{loc}}(X\setminus E),\qquad 1\le\rho\le C_0.
\]
Write $\alpha=f^*\omega$. The limit satisfies
\[
\ii\partial\bar\partial(\rho\alpha^k)\wedge\alpha^{n-k-1}=c_k'\rho\alpha^n\quad\text{on }X\setminus E.
\]
Transfer $\rho$ by the isomorphism $f:X\setminus E\to Y\setminus Z$. The resulting function $r$ satisfies
\[
\ii\partial\bar\partial(r\omega^k)\wedge\omega^{n-k-1}=c_k'r\omega^n\quad\text{on }Y\setminus Z,\qquad 1\le r\le C_0.
\]
Proposition~\ref{exten} applies on the smooth manifold $Y$ with its nondegenerate metric $\omega$. It extends $r$ to a smooth function on $Y$, still satisfying $r\ge1$ and the same equation. The uniqueness of the constant in Theorem~\ref{FWWMT} gives $c_k'=c_k(\omega)$.

Every sequence $\epsilon_j\to0$ therefore has a subsequence whose constants converge to $c_k(\omega)$. Since the constants are uniformly bounded, this proves the full limit \eqref{eq:blowup-eigenvalue-limit}.
\end{proof}

\begin{proof}[Proof of Theorem~\ref{thm2}]
It is enough to consider one blow-up $f:X\to Y$. A blow-up along a smooth divisor is an isomorphism, so only centers of complex codimension at least two need to be considered. If a smooth center has components of different codimensions, the argument applies to its disjoint components, and divisorial components can be omitted.

 if $1\in C_k(Y)$, then there exists a Hermitian metric $\omega^+$ on $Y$ with $c_k(\omega^+)>0$. By Proposition~\ref{prop:blowup-convergence}, we have a Hermitian metric $\Omega^+$ on $X$ with $c_k(\Omega^+)>0$. The same arguments give us that  $-1\in C_k(Y)$ implies  $-1\in C_k(X)$.

Assume $C_k(Y)=\{-1,0,1\}$. 
For $k=n-1$, one always has $C_{n-1}(Y)=\{0\}$, so $1\le k \le n-2$.
Choose Hermitian metrics $\omega^+$ and $\omega^-$ on $Y$ with $c_k(\omega^+)>0$ and $c_k(\omega^-)<0$. Proposition~\ref{prop:blowup-convergence} gives Hermitian metrics $\Omega^+$ and $\Omega^-$ on $X$ with
\[
c_k(\Omega^+)>0,\qquad c_k(\Omega^-)<0.
\]
The convex path $\Omega_t=(1-t)\Omega^-+t\Omega^+$ consists of Hermitian metrics. By the continuity of $c_k$ from \cite[Proposition 9]{FWW}, there is $t_0\in(0,1)$ for which $c_k(\Omega_{t_0})=0$. Therefore $C_k(X)=\{-1,0,1\}$. Repeating the argument proves the result for a finite sequence of blow-ups with smooth centers.
\end{proof}

\section*{Acknowledgements}
This work was supported by 
the National Natural Science Foundation of China (No.12141104, 12371062 and 12431004) and the Natural
Science Foundation of Jiangsu Province (No.BK20251420).

\end{document}